\documentclass[reqno]{amsart}
\usepackage{graphicx, amsmath, parskip, geometry}
\usepackage[english]{babel}
\usepackage{blindtext, layouts, enumitem}
\usepackage{eufrak, amssymb}
\usepackage{array, tabularx, multirow, longtable, xcolor}
\usepackage[export]{adjustbox}
\usepackage{subcaption, wrapfig}
\usepackage[pdftex]{pict2e}
\usepackage{tikz-cd}

\usepackage{centernot}

\usepackage{quiver}

\usepackage{hyperref}
\usepackage{color}

\newcommand{\todo}[1]{{\color{red}#1}}

\newcommand{\A}{\mathcal{A}}

\newcommand{\N}{\mathcal{N}}

\newcommand{\Nnat}{\mathbb{N}}
\newcommand{\Zint}{\mathbb{Z}}

\newcommand{\Cpx}{\mathbb{C}}

\newcommand{\bbK}{\mathbb{K}}

\DeclareMathOperator{\ext}{Ext}
\DeclareMathOperator{\gk}{GKdim}

\numberwithin{equation}{section}

\theoremstyle{plain}
\newtheorem{thm}[equation]{Theorem}
\newtheorem{cor}[equation]{Corollary}
\newtheorem{lem}[equation]{Lemma}
\newtheorem{prop}[equation]{Proposition}
\newtheorem{question}[equation]{Question}

\theoremstyle{definition}
\newtheorem{de}[equation]{Definition}

\newtheorem{rem}[equation]{Remark}

\newtheorem{conj}{Conjecture}

\title{$N$-Koszul algebras of finite global dimension for $N\geq 3$}
\author{So Nakamura}
\email{sonakamura@unr.edu}
\address{Department of Mathematics and Statistics\\ University of Nevada, Reno \\
105B Orvis Building\\ Reno, NV 89512 \\ USA}

\date{\today}

\begin{document}
\keywords{$N$-Koszul algebra, Artin-Schelter regular algebra, noncommutative algebraic geometry}
\maketitle

\begin{abstract}
    Let $N\geq 3$. The class of $N$-Koszul AS regular algebras, or more generally, that of $N$-Koszul AS Gorenstein algebras, has attracted much attention from algebraists. Nevertheless, there have been no known examples of $N$-Koszul AS regular algebras of finite global dimension other than the ones of global dimension $3$. A recent work by Kabbaj showed that, such an $N$-Koszul algebra $A$ of finite global dimension has to have a large global dimension and that $N$ has to be prime, under
    the assumptions that $(1)$ $A$ has a Hilbert series of weighted polynomial rings and that $(2)$ the trivial $A$-module $\mathbb{K}$ has a finite free resolution. All AS regular algebras satisfy the latter assumption and are expected to do the former as well. In this paper, we prove that such an $N$-Koszul algebra $A$ must be one of the known $3$-Koszul AS regular algebras of global dimension $3$ if the order of the pole of its Hilbert series $h_A(t)$ at $t=1$ is greater than $\frac{21d+1}{22}$, where $d$ is the global dimension of $A$. As a corollary, we prove that any $N$-Koszul AS regular algebra $A$ must be one of the known $3$-Koszul AS regular algebras of global dimension $3$ if $A$ has a Hilbert series of weighted polynomial rings and if the GK dimension of $A$ coincides with the global dimension of $A$, both of which have been conjectured to hold for any AS regular algebras. 
\end{abstract}

\setcounter{tocdepth}{2}
\makeatletter
\def\l@subsection{\@tocline{2}{0pt}{2.5pc}{5pc}{}}
\def\l@subsubsection{\@tocline{2}{0pt}{5pc}{7.5pc}{}} 
\makeatother
\tableofcontents

\setcounter{section}{-1}

\section{Introduction}
A polynomial ring over a field is a fundamental ingredient for algebraic geometry. Their grading structure is crucial for the construction of projective varieties. In the area of noncommutative algebraic geometry, which seeks algebro-geometric phenomena in noncommutative ring theory, it is natural to ask for noncommutative analogs of polynomial rings -- especially from the perspective of homological algebra. One might wonder if the noncommutative polynomial rings serve as their noncommutative generalization from this viewpoint. However, they are known to lack nice properties: if they have multiple variables, they are not Noetherian, and they all have global dimension $1$ (this follows from \cite[Corollary 2.5]{Bergman}, for example), unlike the commutative case. One of the classes of rings that behave homologically nicely is that of \emph{Artin-Schelter(AS) regular algebras} introduced in \cite{ArtinSchelter}:

\begin{de}\cite{RogalskiAS}
    Let $A=\bbK\oplus A_1\oplus A_2\oplus\cdots$ be a finitely generated graded algebra over $\bbK$. The algebra $A$ is called \emph{Artin-Schelter(AS) regular} if it has the following properties:
    \begin{enumerate}
        \item $A$ has finite global dimension $d$, 
        \item $A$ has finite GK-dimension, i.e., there exist $c, d>0$ such that $\dim_\bbK A_n\leq cn^d$ for all $n\geq1$.
        \item $A$ is Gorenstein, meaning that
        $$\ext^q_A(k, A)\simeq\left\{\begin{array}{llll}
        0 & \text{if}\ q\neq d\\
        \bbK & \text{if}\ q=d.
        \end{array}\right.$$
    \end{enumerate}
\end{de}

The above definition is indeed equivalent to the original definition in \cite{ArtinSchelter}, see \cite[Remark 2.13]{RogalskiNPC}.
The authors of the paper \cite{ArtinSchelter} attempted to classify AS regular algebras of global dimension 3 by using algebraic geometry. The full classification of those algebras was finally completed by Artin, Tate, and Van den Bergh in \cite{ArtinTateVandenBerghelliptic}. Motivated by the work in \cite{ArtinSchelter}, Berger in \cite{Berger} introduced the notion of \emph{N-Koszul algebras} for $N\geq 2$ based on the two types of the minimal projective resolution of the trivial left module $\bbK$ obtained in \cite{ArtinSchelter}: 

\begin{de}[$N$-Koszul algebras]\cite[Definition 2.10]{Berger}
    Let $N\geq2$ and $A=\bbK\oplus A_1\oplus A_2\oplus\cdots$ be a locally finite graded algebra generated in degree 1 over $\bbK$. The algebra $A$ is called \emph{$N$-Koszul} if the trivial left $A$-module $\bbK$ has a minimal projective resolution of the form
\[\begin{tikzcd}[column sep=small]
	\cdots & {P^i} & \cdots & {P^1} & {P^0} & k & 0
	\arrow[from=1-1, to=1-2]
	\arrow[from=1-2, to=1-3]
	\arrow[from=1-3, to=1-4]
	\arrow[from=1-4, to=1-5]
	\arrow[from=1-5, to=1-6]
	\arrow[from=1-6, to=1-7]
\end{tikzcd}\]
    where each $P^i$ is generated in degree $\nu_i$ where
    $$\nu_i=\left\{\begin{array}{llll}
        \frac{i}{2}N & \text{if}\ i\ \text{is even}\\
        \frac{i-1}{2}N+1 & \text{if}\ i\ \text{is odd}.
        \end{array}\right.$$
\end{de}

The definition coincides with that of the usual Koszul algebras when $N=2$. Following the definition, AS regular algebras of global dimension 3 are either (2-)Koszul or 3-Koszul. 

$N$-Koszul AS regular algebras as well as $N$-Koszul AS \emph{Gorenstein} algebras, which are a generalization of the former, have been studied intensively in the context of algebras associated to homogeneous $N$-linear forms. Berger and Marconnet showed in \cite[Theorem 1.2]{BergerMarconnet} that an $N$-Koszul algebra of finite global dimension is AS Gorenstein if and only if its Yoneda algebra is Frobenius. Dubois-Violette showed in \cite[Theorem 11]{DuboisViolette} that every 
$N$-Koszul AS Gorenstein algebra of finite global dimension is isomorphic to the algebra $\A(\omega, N)$ described in \cite[Section 5]{DuboisViolette} where $\omega$ is a \emph{preregular}(\cite[Definition 2]{DuboisViolette}) linear form. This algebra is a special case of the \emph{derivation-quotient algebra} of some twisted superpotential introduced in \cite{BocklandtSchedlerWemyss}, where the result \cite[Theorem 11]{DuboisViolette} is generalized (\cite[Theorem 6.2]{BocklandtSchedlerWemyss}) to include the class of $N$-Koszul \emph{(twisted) Calabi-Yau algebras}, which is a non-connected generalization of that of $N$-Koszul AS regular algebras. Other studies on $N$-Koszul AS regular algebras include \cite{MoriSmith}, where the homological determinants of automorphisms on such algebras are discussed, and \cite{ChirvasituWaltonWang}, which studies Hopf algebras over them constructed using $N$-linear forms.

In spite of the extensive research on AS Gorenstein algebras in recent years, those classified in \cite{ArtinTateVandenBerghelliptic} are the only known examples of $N$-Koszul AS regular algebras for $N\geq 3$ so far. This fact motivated Kabbaj to study the nonexistence of other such algebras in \cite{Kabbaj}. In the paper, the author shows the following theorem:

\begin{thm}\cite[Theorem 1.2]{Kabbaj}\label{Kabbajmain}
Let $N\geq 3$ and $A$ be an $N$-Koszul graded algebra of finite global dimension $d$ whose Hilbert series is that of a weighted polynomial ring. If $\bbK$ has a finite free resolution, then:
\begin{enumerate}
    \item(\cite[Proposition 3.1, 3.5]{Kabbaj}) The global dimension is $d=2k+1$ where $k$ is an odd integer.
    \item(\cite[Corollary 3.9]{Kabbaj}) The global dimension satisfies $d\geq\frac{2^N-4}{N}+1$.
    \item(\cite[Theorem 3.7]{Kabbaj}) $N$ is prime.
\end{enumerate}
\end{thm}

Here, the algebra $A$ is assumed to satisfy two assumptions: it is assumed to have a Hilbert series of a weighted polynomial ring (\ref{weighted}) --  we will call it the \emph{Hilbert series hypothesis} -- and the trivial $A$-module $\bbK$ is to have a finite free resolution. These two assumptions are widely believed to hold for all AS regular algebras and are therefore considered mild, since our main interest is the existence of $N$-Koszul AS regular algebras (for $N\geq 3$). The above theorem supports the following conjecture:

\begin{conj}\cite[Conjecture 1.3]{Kabbaj}\label{Kabbajconj}
    Let $N\geq3$ and $A$ be an $N$-Koszul graded algebra of finite global dimension $d$ whose Hilbert series is that of a weighted polynomial ring. If $\bbK$ has a finite free resolution, then $A$ is a 3-Koszul AS regular algebra of global dimension 3.
\end{conj}

The key idea of the proof of Theorem \ref{Kabbajmain} is to consider the relationship between the \emph{characteristic polynomial} of the trivial left $A$ module $\bbK$, which is defined thanks to the assumption that $\bbK$ has a finite free resolution, and the Hilbert series $h_A(t)$ of $A$, and then apply the Hilbert series hypothesis \ref{weighted}: if $p(t)$ denotes the characteristic polynomial of the trivial module $\bbK$, then $h_A(t)=\frac{1}{p(t)}$ and $p(t)$ can be written in the following ways(\cite[remark 3.2, (3.3.2), (3.3.1)]{Kabbaj}):
\begin{equation}\label{Kabbajeq1}
p(t)=1-\beta_1+\beta_2t^N-\beta_3t^{N+1}+\cdots-\beta_2t^{(k-1)N+1}+\beta_1t^{kN} -t^{kN+1}
\end{equation}
and
\begin{equation}\label{Kabbajeq2}
    p(t)=\prod_{i=1}^m(1-t^i)^{n_i}
\end{equation}
where $k$, $\beta_i, 1\leq i\leq k$, and $m$ are all positive integers, $n_i$'s are nonnegative integers, and $n_m>0$ (the term “positive" in \cite[remark 3.2]{Kabbaj} should be “non-negative" instead). Letting $$\beta_0:=1, \beta_{2k+1-i}:=\beta_{i}\ \text{for all}\ 0\leq i\leq k\ \text{and}\ q(t):=\sum^k_{i=0}\beta_{2i}t^{iN},$$ 
the expression \ref{Kabbajeq1} can also be written as
$$p(t)=q(t)-t^{kN+1}q(1/t).$$
Assuming that the equations \ref{Kabbajeq1} and \ref{Kabbajeq2} hold, $N$ is shown to be prime by leading to a contradiction from the assumption that $N$ has a nontrivial divisor (\cite[Theorem 3.7]{Kabbaj}).

The goal of this paper is to show that, under the assumption that $N$ is an odd prime and that $\sum^m_{i=1}n_i>\frac{21}{11}k+1$, the above equalities in fact lead to a contradiction unless $N=3$ and $k=1$. For a simpler notation, we use $\alpha_i$ instead of $\beta_i$. In the following main theorem, the coefficients $\alpha_i$ correspond to the coefficients $\beta_{2i}$ in the above setting.

\begin{thm}\label{main}
    Let $q(t)\in\Zint[t]$ be a polynomial of the form
    $$q(t)=\sum^k_{i=0}\alpha_it^{iN}$$
    where $N$ is an odd prime, $k$ is a positive integer, $\alpha_0=1$ and the coefficients $\alpha_i$'s are all positive. Set
    $$p(t):=q(t)-t^{kN+1}q(1/t)\in\Zint[t].$$
    Suppose that
    \begin{enumerate}
        \item the polynomial $p(t)$ can be written in the form
    $$p(t)=\prod^m_{i=1}(1-t^i)^{n_i}$$
    where $n_i\geq0$ for all $1\leq i\leq m$, $n_m>0$, and
       \item $\sum_{i=1}^mn_i>\frac{21}{11}k+1$.
    \end{enumerate}
     then $N=3$, $k=1$, $p(t)=(1-t)^2(1-t^2)$ and $q(t)=1+2t$.
\end{thm}

The fraction $\frac{21}{11}$ is related to  inequality $(**)$ that appears in the proof of Corollary \ref{7m<2N-1}. We might be able to improve this result by taking a smaller coefficient of $k$ and weakening the assumption $(2)$. However, the author thinks that the improvement of the coefficient possibly requires a significant amount of computation. 

Note that the sum $\sum_{i=1}^mn_i$ is at most $2k+1$ by \emph{Descartes' rule of signs}.
The first assumption $(1)$ corresponds to the assumption in Theorem \ref{Kabbajmain} that the $N$-Koszul algebra $A$ has a Hilbert polynomial of a weighted polynomial ring. If the $N$-Koszul algebra $A$ is AS regular, the second assumption means that the difference between the global dimension and the GK dimension is not too large. Indeed, $2k+1=d$ is the global dimension of $A$ as in Theorem \ref{Kabbajmain}, whereas the sum $\sum_{i=1}^m n_i$ corresponds to the order of the pole of the Hilbert series $h_A(t)$ at $t=1$, which is equal to the GK dimension of the algebra (\cite[Proposition 2.21]{ArtinTateVandenBergh}, \cite[Corollary 2.2]{StephensonZhang}). The equality between those dimensions is believed to hold for all AS regular algebras, and hence this assumption is considered to be mild. Together with \cite[Remark 3.2]{Kabbaj}, Theorem \ref{main} implies that Conjecture \ref{Kabbajconj} is true under an assumption on the order of the pole of $h_A(t)$ at $t=1$:


\begin{cor}\label{NKcor}\ref{Ncor}
    Let $N\geq 3$ and $A$ be an $N$-Koszul algebra of finite global dimension $d$ whose trivial module $\bbK$ has a finite free resolution.
    If the Hilbert series of $A$ is that of a weighted polynomial ring and the order of the pole of its Hilbert series $h_A(t)$ at $t=1$ is greater than $\frac{21d+1}{22}$, then $A$ is a $3$-Koszul AS regular algebra of global dimension $3$.
\end{cor}


Furthermore, the trivial left $A$-module $k$ has a finite free resolution as long as $A$ satisfies the Gorenstein condition of AS regular algebras (\cite[Proposition 3.1]{StephensonZhang}). Therefore:

\begin{cor}\label{ASRcor}\ref{AScor}
    Let $N\geq 3$ and $A$ be an $N$-Koszul AS regular algebra over $k$ of finite global dimension $d$. If the Hilbert series of $A$ is that of a weighted polynomial ring and $A$ has GK dimension $d$, then $A$ is a $3$-Koszul AS regular algebra of global dimension $3$. 
\end{cor}

We close the introduction by posing the following natural question:

\begin{question}
    Does the assertion of Theorem \ref{main} still hold without Assumption $(2)$?
\end{question}

Note that if the answer to this question is “Yes", then Conjecture \ref{Kabbajconj} is true.

\section{Background}

We first briefly review some basic notions in Subsection \ref{prelim}. In the first half of Subsection \ref{remarkmain}, we recall how to deduce the equalities \ref{Kabbajeq1} and \ref{Kabbajeq2} from the assumption of Conjecture \ref{Kabbajconj}. Then we prove Corollaries \ref{NKcor} and \ref{ASRcor} using Theorem \ref{main} in the second half of Subsection \ref{remarkmain}. We keep the notations used in \cite{Kabbaj} as much as possible. 

\subsection{Preliminaries}\label{prelim}

See \cite{Kabbaj}, \cite{McConnellRobson}, \cite{PositselskiPolishchuk}, \cite{ReyesRogalski}, \cite{RogalskiAS} and \cite{StephensonZhang} for a more detailed explanation of the contents of this subsection. 

Throughout the paper, we fix a field $\bbK$. All vector spaces are $\bbK$-vector spaces. A $\bbK$-algebra $A$ is called \emph{$\Nnat$-graded} if $A$ is equipped with a decomposition $A=A_0\oplus A_1\oplus A_2\oplus\cdots$ as a vector space such that $A_iA_j\subset A_{i+j}$ for all $i, j\geq0$. The elements in $A_n$ are said to be of \emph{degree $n$}. The algebra $A$ is said to be 
\begin{enumerate}
    \item \emph{locally finite} if $\dim_\bbK A_n<\infty$ for all $n\geq0$,
    \item \emph{connected} if $A_0=\bbK$, and
    \item \emph{generated in degree one} if $A$ is generated by the elements in $A_1$ as a $\bbK$-algebra.
\end{enumerate}
\textbf{All graded algebras are assumed to satisfy these three conditions throughout the paper.}

Let $A$ be such an algebra. A left $A$-module $M$ is called \emph{$\Zint$-graded} if $M$ is equipped with a decomposition $M=\bigoplus_{n\in\mathbb{Z}}M_n$ as a vector space such that $A_iM_j\subset M_{i+j}$ for all $i\geq0$ and $j\in\mathbb{Z}$. Graded right $A$-modules are defined similarly. All modules over an $\Nnat$-graded algebra in this paper will be left modules and $\Zint$-graded. For a graded module $M$ and $n\in\mathbb{Z}$, we denote by $M(n)$ the graded module isomorphic to $M$ as an $A$-module whose grading is defined by $M(n)_i:=M_{i+n}$ for all $i\in\mathbb{Z}$.

A graded $A$-module $M$ is said to be 
\begin{enumerate}
    \item \emph{locally finite} if $\dim_\bbK M_n<\infty$ for all $n\in\mathbb{Z}$.
    \item \emph{left bounded} if there exists $n\in\mathbb{Z}$ such that $M_i=0$ for all $i<n$.
\end{enumerate}
For such graded module $M$, the \emph{Hilbert series} $h_M(t)\in\Zint[[t, t^{-1}]]$(\cite{ArtinTateVandenBergh}, \cite{StephensonZhang}, \cite{PositselskiPolishchuk}) is defined by
$$h_M(t):=\sum_{i\in\mathbb{Z}}(\dim_\bbK M_i)t^i$$
The $A$-module $A_+:=A_1\oplus A_2\oplus\cdots$ is called the \emph{irrelevant ideal} and the quotient module $A/A_+\simeq A_0=\bbK$ is called the \emph{trivial} $A$-module. The trivial module is locally finite and left bounded. By definition, $h_\bbK(t)=1$. The Hilbert series of a graded algebra $A$ is that of $A$ seen as a graded $A$-module.

\begin{de}[Hilbert series of a weighted polynomial ring]\label{weighted}\cite[Definition 2.3]{Kabbaj}
    A graded algebra $A$ is said to have the \emph{Hilbert series of a weighted polynomial ring} if there exist nonnegative integers $n_i$ and positive integer $m$ such that
    $$h_A(t)=\prod_{i=1}^m\frac{1}{(1-t^i)^{n_i}}.$$
\end{de}

Next, we review some definitions of dimensions defined for graded algebras. Let $A$ be a graded algebra and $M$ be a graded left $A$-module. Recall that the \emph{left projective dimension} of $M$ is the minimal (possibly infinite) length of its (ungraded) projective resolutions. The \emph{left global dimension} of $A$ is the supremum of the left projective dimensions of its (ungraded) left modules. Similarly, the \emph{graded left projective dimension} of $M$ is the minimal (again, possibly infinite) length of its \emph{graded} projective resolutions, and the \emph{graded left global dimension} of $A$ is the supremum of the \emph{graded} left projective dimensions of its \emph {graded} left modules. The \emph{right projective dimension} and the \emph{graded right projective dimension} of a graded right $A$-module, and furthermore the \emph{right global dimension} and the \emph{graded right global dimension} of $A$ are defined similarly. In our case where $A$ is connected, these four kinds of global dimensions of $A$ are all equal to the (left or right) projective dimension of the (ungraded) $A$-module $\bbK$ by \cite[Proposition 3.18(3)]{ReyesRogalski}. Therefore, they will simply be called the \emph{global dimension} of $A$.

There is another kind of dimension defined for (ungraded) modules over (ungraded) algebras, namely the \emph{Gelfand-Kirillov(GK) dimension}. The GK dimension of an algebra is that of itself seen as a left module. In this paper, we will only use the GK dimensions of algebras, so we do not state the general definition of the GK dimension. However, if $A$ is a graded algebra in our setting, then the GK dimension $\gk(A)$ is given by
$$\gk(A)=\limsup_{n\rightarrow\infty}\log_nd(n)$$
where $d(n)$ denotes the dimension of the subspace $\bbK\oplus A_1\oplus\cdots\oplus A_n\subset A$. Thus the Hilbert series of a graded module measures the growth of the subspaces in this form. 
If the Hilbert series of $A$ is nice enough, then it tells us the GK dimension; indeed, \cite[Proposition 2.21]{ArtinTateVandenBergh} and \cite[Corollary 2.2]{StephensonZhang} imply that if $A$ is AS regular and has a Hilbert series of a weighted polynomial ring, then the GK dimension of $A$ coincides with the order of the pole of $h_A(t)$ at $t=1$.


If the graded module $M$ is nice enough, a characteristic polynomial is also defined by using its minimal free resolution. Let $A$ be a graded algebra with finite global dimension $d$ and $M$ be a graded left $A$-module. By \cite[p. 339]{ArtinTateVandenBergh}, or by \cite[Lemma 2.6]{MinamotoMori} and \cite[Section 2.1]{ReyesRogalski}, every graded left bounded projective $A$-module is of the form $\oplus_i A(l_i)$ since $A$ is connected. Thus $M$ admits an augmented minimal free resolution of the form
\[\begin{tikzcd}[column sep=small]
	0 & {\bigoplus^{\beta_d}_{i=1}A(-l^d_i)} & \cdots & {\bigoplus^{\beta_1}_{i=1}A(-l^1_i)} & {\bigoplus^{\beta_0}_{i=1}A(-l^0_i)} & M & 0
	\arrow[from=1-1, to=1-2]
	\arrow[from=1-2, to=1-3]
	\arrow[from=1-3, to=1-4]
	\arrow[from=1-4, to=1-5]
	\arrow[from=1-5, to=1-6]
	\arrow[from=1-6, to=1-7]
\end{tikzcd}\]
    where $\beta_0, \cdots, \beta_d$ are some (possibly infinite) cardinals. The module $M$ is said to have a \emph{finite free resolution} if $\beta_0, \cdots, \beta_d$ are all finite.

\begin{de}[characteristic polynomial, \cite{ArtinTateVandenBergh}, \cite{StephensonZhang}]
     If $M$ has a finite free resolution, the \emph{characteristic polynomial} $c_M(t)$ of $M$ is defined by
    $$c_M(t):=\sum^d_{j=0}(-1)^i\left(\sum^{\beta_j}_{i=1}t^{l^j_i}\right).$$
\end{de}

The following lemma relates the Hilbert series of an algebra to that of a module. 

\begin{lem}\cite[Lemma 2.3]{StephensonZhang}\label{2.3}
Suppose $M$ has a finite free resolution and let $c_M(t)$
be the characteristic polynomial of $M$. Then $h_M(t)=c_M(t)h_A(t)$.
\end{lem}


We close this subsection by recalling the definition of AS regular algebras, which are the main interest of the project:

\begin{de}\cite{RogalskiAS}
    Let $A=\bbK\oplus A_1\oplus A_2\oplus\cdots$ be a finitely generated graded algebra over $\bbK$. The algebra $A$ is called \emph{Artin-Schelter(AS) regular} if it has the following properties:
    \begin{enumerate}
        \item $A$ has finite global dimension $d$, 
        \item $A$ has finite GK-dimension, i.e., there exist $c, d>0$ such that $\dim_\bbK A_n\leq cn^d$ for all $n\geq1$.
        \item $A$ is Gorenstein, meaning that
        $$\ext^q_A(k, A)\simeq\left\{\begin{array}{llll}
        0 & \text{if}\ q\neq d\\
        \bbK & \text{if}\ q=d.
        \end{array}\right.$$
    \end{enumerate}
\end{de}

AS regular algebras are expected to satisfy a number of nice properties, some of which we list as follows:

\begin{conj}\cite{ArtinTateVandenBergh}, \cite{RogalskiAS}
Let $A$ be an AS regular algebra of global dimension $d$. Then the following should hold:
    \begin{enumerate}
        \item $A$ is Noetherian.
        \item $A$ is a domain.
        \item The GK dimension of $A$ equals $d.$
        \item(\cite{PositselskiPolishchuk}) $A$ has the Hilbert series of a weighted polynomial ring.
    \end{enumerate}
\end{conj}

Again, Corollary \ref{ASRcor} states that if $(3)$ and $(4)$ above are true, then the $3$-Koszul AS regular algebras of global dimension $3$ are the only $N$-Koszul AS regular algebras for $N\geq3$.

\subsection{Kabbaj's remark \cite[2.2]{Kabbaj} and the consequence of Theorem \ref{main}}\label{remarkmain}

In this subsection, we briefly recall how to obtain the equalities \ref{Kabbajeq1} and \ref{Kabbajeq2}, assuming the existence of an $N$-Koszul algebra satisfying the assumptions of \ref{Kabbajmain}. For a reference, see also \cite{Kabbaj}. After that, we prove Corollaries \ref{NKcor} and \ref{ASRcor}. \textbf{Throughout the paper, the integer $N$ will always be greater than $2.$} Let us recall the definition of $N$-Koszul algebras:

\begin{de}[$N$-Koszul algebras]\cite[Definition 2.10]{Berger}
    Let $A=\bbK\oplus A_1\oplus A_2\oplus\cdots$ be a locally finite graded algebra generated in degree 1 over $\bbK$. The algebra $A$ is called \emph{$N$-Koszul} if the trivial left $A$-module $\bbK$ has a minimal projective resolution of the form
\[\begin{tikzcd}[column sep=small]
	\cdots & {P^i} & \cdots & {P^1} & {P^0} & k & 0
	\arrow[from=1-1, to=1-2]
	\arrow[from=1-2, to=1-3]
	\arrow[from=1-3, to=1-4]
	\arrow[from=1-4, to=1-5]
	\arrow[from=1-5, to=1-6]
	\arrow[from=1-6, to=1-7]
\end{tikzcd}\]
    where each $P^i$ is generated in degree $\nu_i$ where
    $$\nu_i=\left\{\begin{array}{llll}
        \frac{i}{2}N & \text{if}\ i\ \text{is even}\\
        \frac{i-1}{2}N+1 & \text{if}\ i\ \text{is odd}.
        \end{array}\right.$$
\end{de}

Now suppose that $A$ is an $N$-Koszul algebra of finite global dimension $d$ such that
\begin{enumerate}
    \item the Hilbert series is that of a weighted polynomial ring, and
    \item the trivial $A$-module $\bbK$ has a finite free resolution.
\end{enumerate}
The Hilbert series hypothesis $(1)$ above and $c_\bbK(t)h_A(t)=h_\bbK(t)=1$ from Lemma \ref{2.3} imply that the characteristic polynomial $p(t)$ of the trivial module $\bbK$ must be of the form
$$p(t):=c_\bbK(t)=\prod_{i=1}^m(1-t^i)^{n_i},$$
which is Equation \ref{Kabbajeq2}.
On the other hand, the minimal projective resolution is of the form 
\[\begin{tikzcd}[column sep=small]
	0 & {\bigoplus^{\beta_d}_{i=1}A(-\nu_d)} & \cdots & {\bigoplus^{\beta_1}_{i=1}A(-\nu_1)} & A & \bbK & 0
	\arrow[from=1-1, to=1-2]
	\arrow[from=1-2, to=1-3]
	\arrow[from=1-3, to=1-4]
	\arrow[from=1-4, to=1-5]
	\arrow[from=1-5, to=1-6]
	\arrow[from=1-6, to=1-7]
\end{tikzcd}\]
where $\beta_0=1.$ Therefore we have
$$p(t):=c_\bbK(t)=\sum^d_{j=0}(-1)^j\left(\sum^{\beta_j}_{i=1}t^{\nu_i}\right)=\sum^d_{i=0}(-1)^i\beta_it^{\nu_i}.$$
Note that the degree of $p(t)$ is $\nu_d$. 

If $d$ were even, then the coefficient
of $t^{\nu_d}$ is $\beta_d$. On the other hand, we have $t^{\nu_d}p(1/t)=(-1)^{\nu_d}p(t)$ by Equation \ref{Kabbajeq2} and hence $1=(-1)^{\nu_d}\beta_d$ by comparison of the coefficient of $t^{\nu_d}$. Since $\beta_d>0$, $\nu_d$ must be even and $t^{\nu_d}p(1/t)=p(t)$. However, the comparison of the coefficient of $t^{\nu_1}=t$ yields $0=-\beta_1$ since $N\geq 3$. This contradicts $\beta_2>0$ and hence $d$ must be odd(this is the main idea of the proof of \cite[Proposition 3.1]{Kabbaj}). 

The coefficient of $t^{\nu_d}$ is then $-\beta_d$ and we obtain $1=(-1)^{\nu_d+1}\beta_d$. Therefore $t^{\nu_d}p(1/t)=-p(t)$. Writing $d=2k+1$, we have
$$p(t)=1-\beta_1t+\beta_2t^N-\cdots-\beta_{2k-1}t^{(k-1)N+1} +\beta_{2k}t^{kN}-\beta_{2k+1}t^{kN+1},$$
so that $\beta_{i}=\beta_{2k+1-i}$ for all $1\leq i\leq 2k+1$.
We may write $p(t)$ in the form
$$p(t)=1-\beta_1t+\beta_2t^N-\beta_3t^{N+1}+\cdots-\beta_2t^{(k-1)N+1}+\beta_1t^{kN} -t^{kN+1},$$
which is Equation \ref{Kabbajeq1}. If $d=1$, then we have $p(t)=1-t$. Since $h_A(t)=\frac{1}{p(t)}$, we obtain $\dim_\bbK A_i=1$ for all $i\geq 0$. This implies that $A$ is generated by a single element as an algebra over $\bbK$ and $\dim_\bbK A=\infty$. Thus $A\simeq\bbK[x]$, which is ($2$-)Koszul. Therefore $d>1$ and $k$ must be positive. 

Letting $$\alpha_0:=1, \alpha_i:=\beta_{2i}\ \text{for all}\ 0\leq i\leq k\ \text{and}\ q(t):=\sum^k_{i=0}\alpha_{i}t^{iN},$$ 
we obtain
$$p(t)=q(t)-t^{kN+1}q(1/t)$$
as we have seen just before Theorem \ref{main} in the introduction.

Now we are ready to prove Corollaries \ref{NKcor} and \ref{ASRcor}.

\begin{cor}\label{Ncor}\ref{NKcor}
    Let $N\geq 3$ and $A$ be an $N$-Koszul algebra of finite global dimension $d$ whose trivial module $\bbK$ has a finite free resolution.
    If the Hilbert series of $A$ is that of a weighted polynomial ring and the order of the pole of its Hilbert series $h_A(t)$ at $t=1$ is greater than $\frac{21d+1}{22}$, then $A$ is a $3$-Koszul AS regular algebra of global dimension $3$.
\end{cor}
\begin{proof}
    Let $A$ be such an algebra and $p(t)$ be the characteristic polynomial of the trivial $A$-module $\bbK$. By \cite[Theorem 1.2(3)]{Kabbaj}, the integer $N$ is prime and there is a polynomial $q(t)$ satisfying the assumption of Theorem \ref{main} for $k=\frac{d-1}{2}$. The order of the pole of $h_A(t)$ at $t=1$ is equal to $\sum n_i$ and this is greater than $\frac{21d+1}{22}=\frac{21}{11}k+1$ by assumption and hence $p(t)$ satisfies conditions $(1)$ and $(2)$ of Theorem \ref{main}. It implies that $N=3$ and $k=1$ so that $d=3$. 
\end{proof}

In particular, we obtain the following result:

\begin{cor}\label{AScor}\ref{ASRcor}
    Let $N\geq 3$ and $A$ be an $N$-Koszul AS regular algebra over $\bbK$ of finite global dimension $d$. If the Hilbert series of $A$ is that of a weighted polynomial ring and $A$ has GK dimension $d$, then $A$ is a $3$-Koszul AS regular algebra of global dimension $3$. 
\end{cor}
\begin{proof}
    \cite[Proposition 3.1]{StephensonZhang} shows that the trivial $A$-module $\bbK$ has a finite free resolution. Also, the GK dimension $d$ of $A$ is greater than $\frac{21d+1}{22}$ since $d>1$. Thus the assertion follows from Corollary \ref{Ncor}. 
\end{proof}

\section{Proof of Theorem \ref{main}}

We break the proof of Theorem \ref{main} into five subsections: the first subsection is devoted to some computational results regarding the power sums of the roots of $p(t)$. This part only uses Assumption $(1)$ and does not use Assumption $(2)$. The results of this subsection will be used throughout the proof of Theorem \ref{main}. The second subsection uses Assumption $(2)$ to bound the integer $k$ of Theorem \ref{main} by using $N$ and $n_1$, which corresponds to the number of generators of $A$ in Kabbaj's formulation \cite[Remark 2.2]{Kabbaj}. The result of this subsection will be useful for the case where $m$ is relatively large compared to $N$, and it will be used in \ref{7m>2N-1imp} and \ref{53m>N-1}. When $N$ is sufficiently large and $m$ is relatively small (\ref{Na2}), we use Nagura's generalization of 
\emph{Bertrand's postulate} or
the \emph{Bertrand-Chebyshev theorem}, which is a well-known number theoretic fact on the distribution of primes conjectured by Bertrand and proved by Chebyshev.
\begin{thm}\cite{Nagura}
    Let 
    $$a_1:=2, a_2:=8, a_3:=9, a_4:=24, a_5:=25.$$
    For every integer $1\leq n\leq 5$ and every real number $x\geq a_n$, there exists a prime number $l$ such that
    $$x<l<\left(\frac{n+1}{n}\right)x.$$
\end{thm}


Note that the classical Bertrand's postulate is the case $n=1$.

Before we start the proof, let us recall the setting of Theorem \ref{main}: we have a polynomial
$$q(t)=\sum^k_{i=0}\alpha_it^{iN}=\alpha_0+\alpha_1t^N+\cdots+\alpha_kt^{kN}$$
where $N$ is an odd prime, $k$ is a positive integer, $\alpha_0=1$ and $\alpha_1,\cdots,\alpha_k$ are positive integers. We assume that the polynomial $p(t)$ defined by
$$p(t):=q(t)-t^{kN+1}q(1/t)=\alpha_0-\alpha_kt+\alpha_1t^N+\cdots-\alpha_1t^{(k-1)N+1}+\alpha_kt^{kN}-\alpha_0t^{kN+1}$$
can be written as 
$$p(t)=\prod^m_{i=1}(1-t^i)^{n_i}$$
where $m, n_1, \cdots, n_m$ are nonnegative integers and $m, n_m>0$. Note that $n_1>0$ since $p(1)=0$. We already know that $n_{lN}=0$ for all $l\geq 1$ such that $1\leq lN\leq m$ by \cite[Proposition 3.3]{Kabbaj}. This fact can also be proved by using the equality $(1-\zeta)q(t)=p(\zeta t)-\zeta p(t)$ and $q(1)\neq0$ where $\zeta$ is the primitive $N$-th root of unity. 



\subsection{Computation of the sequence $\{\pi_i\}_{i\geq1}$}

In this subsection, we consider a sequence defined as follows:

\begin{de}\label{piseq}
For every $i\geq 1$, we set
$$\pi_i:=\sum_{1\leq d|i}dn_d,$$
where $d$ runs over all the positive divisors of $i$ and if $i>m$ we interpret $n_i=0$. 
\end{de}

This $\pi_i$ is in fact a power sum of the roots of $p(t)$ counted with multiplicity. Indeed,
$$\sum_{\zeta^d=1}\zeta^i=\sum^{d-1}_{j=0}\omega^{ij}=\left\{\begin{array}{llll}
d & \text{if $d$ divides $i$}\\
0 & \text{otherwise}
\end{array}\right.$$
where $\omega$ is a primitive $d$-th root of unity. Hence
$$\sum_{p(\zeta)=0}\zeta^i=\sum^m_{d=1}n_d\left(\sum_{\zeta^d=1}\zeta^i\right)=\sum_{1\leq d|i}n_d\cdot d=\pi_i.$$

To compute $\pi_i$, we first take the logarithm of the equality 
$$p(t)=\prod^m_{i=1}(1-t^i)^{n_i}.$$
to obtain the following equality of power series
$$\log p(t)=\sum^m_{i=1}n_i\log(1-t^i)=-\sum^m_{i=1}n_i\left(\sum_{j=1}^\infty\frac{t^{ij}}{j}\right)=-\sum^\infty_{n=1}\frac{\pi_n}{n}t^n.$$

On the other hand, the logarithmic formula \cite[3.5 Theorem A]{Comtet} shows that
$$\log p(t)=\sum^\infty_{i=1}\frac{L_n}{n!}t^n$$
where
$$L_n=\sum^n_{i=1}(-1)^{i-1}(i-1)!B_{n, i}(g_1, 2g_2, \cdots, (n-i+1)!g_{n-i+1}).$$
Here, $B_{n, i}$ is the (exponential) partial Bell polynomial(\cite[3.3, p.133]{Comtet}) defined by
$$B_{n, i}(x_1, x_2, \cdots, x_{n-i+1}):=\sum\frac{n!}{j_1!j_2!\cdots j_{n-i+1}!}\left(\frac{x_1}{1!}\right)^{j_1}\left(\frac{x_2}{2!}\right)^{j_2}\cdots\left(\frac{x_{n-i+1}}{(n-i+1)!}\right)^{j_{n-k+1}}$$
where $j_i$'s are subject to the conditions
$$j_1+j_2+\cdots+j_{n-i+1}=i,\ j_1+2j_2+\cdots+(n-i+1)j_{n-i+1}=n.$$
The $g_i$'s are the coefficients of $p(t)$ enumerated from the lowest degree to the highest:
$$g_i 
:=\left\{\begin{array}{llll}
0     &  \text{if}\ i\centernot\equiv 0, 1\ \text{or}\ i>kN+1\\
\alpha_l     & i\equiv 0, i=lN\ \text{for some}\ 0< l\leq k\\
-\alpha_{k-l} & i\equiv 1, i=lN+1\ \text{for some}\ 0\leq l\leq k
\end{array}\right.$$
for all $i\geq 1$. Therefore, we have
\begin{eqnarray*}
\pi_n&=&-\frac{L_n}{(n-1)!}\\
&=&\frac{-1}{(n-1)!}\sum^n_{i=1}(-1)^{i-1}(i-1)!B_{n, i}(g_1, 2g_2, \cdots, (n-i+1)!g_{n-i+1})\\
&=&n\sum^n_{i=1}\frac{(-1)^i}{i}
\hat{B}_{n, i}(g_1, g_2, \cdots, g_{n-i+1})
\end{eqnarray*}
where $\hat{B}_{n, i}$ is the ordinary Bell polynomial(\cite[p.136]{Comtet}) given by
$$\hat{B}_{n, i}(x_1, \cdots, x_{n-i+1}):=\sum\binom{i}{j_1, \cdots, j_{n-i+1}}x^{j_1}_1\cdots x^{j_{n-i+1}}_{n-i+1}$$
where $j_i$'s satisfy
$$j_1+j_2+\cdots+j_{n-i+1}=i,\ j_1+2j_2+\cdots+(n-i+1)j_{n-i+1}=n.$$
Since most of the $g_i$'s are $0$, we may write the ordinary Bell polynomials as follows:
$$\hat{B}_{n, i}(g_1, \cdots, g_{n-i+1})=\sum\binom{i}{j_1, j_N, j_{N+1}, \cdots}(-\alpha_k)^{j_1}\alpha^{j_N}_1(-\alpha_{k-1})^{j_{N+1}}\alpha_2^{j_{2N}}\cdots$$
where $j_i$'s satisfy
$$j_1+j_N+j_{N+1}+\cdots=i,\ j_1+Nj_N+(N+1)j_{N+1}+2Nj_{2N}+\cdots=n.$$
The equality on the right considerably reduces the number of possibilities of $j_i$'s when $n$ is small enough. Note that $\pi_n$ can be viewed as a polynomial in $\alpha_j$'s with rational coefficients. Exponential Bell polynomials and ordinary Bell polynomials are known to have  integer coefficients. To simplify the notation, We set $b_{n, i}:=\hat{B}_{n, i}(g_1, \cdots, g_{n-i+1})$ so that
$$\pi_n=n\sum^n_{i=1}\frac{(-1)^i}{i}
b_{n, i}=-nb_{n, 1}+\frac{n}{2}b_{n, 2}-\frac{n}{3}b_{n, 3}+\cdots+(-1)^nb_{n, n}$$

The aim of this subsection is to prove the following:

\begin{prop}\label{seqp}
Let $s:=n_1=\pi_1$.
\begin{enumerate}
    \item We have $\pi_n=s^n$ for all $1\leq n\leq N-1$.
    \item $\alpha_1=\frac{s^N-s}{N}$. In particular, $s>1$.
    \item $\pi_{N+l}=\frac{-ls^{N+l}+(N+l)s^{l+1}}{N}+(N+l)s^{l-1}\alpha_{k-1}$ for all $1\leq l\leq N-1$.
    \item $\pi_{N+l}\geq\frac{N-(l+1)}{2N-1}s^{N+l}$ for all $1\leq l\leq N-1$.
    \item Let $N>3$. For all $N+2<n\leq 2N-1$, $\pi_n$ is divisible by $s^2$.
    \item Let $N>3$. For all $2N+3<n\leq 3N-1$, $\pi_n$ is divisible by $s^2$.

    
\end{enumerate}
\end{prop}

\begin{proof}
$(1)$: Since $n<N$, $j_1=n$ and the other $j_l$'s are all $0$. Therefore $b_{n, i}=0$ if $i<N$. Hence  $\pi_n=(-1)^nb_{n, n}=(-1)^n(-\alpha_k)^n=s^n$.\\
$(2)$: If $n=N$, then $j_1=N$ or $j_N=1$ so that $b_{n, i}=0$ unless $i=1$ or $i=N$:
\begin{center}
\begin{tabular}{ |c|c|c|c|c|c|c|c| }
\hline
    $j_{N}$ & $j_1$ & $i$ & mult. coeff.\\
\hline
   $1$ &  & $1$ & $1$\\ 
\hline
      & $N$ & $N$ & $1$\\ 
\hline
\end{tabular}
\end{center}  
Therefore 
$$\pi_N=N\sum^N_{i=1}\frac{(-1)^i}{i}
b_{N, i}=N\left(-b_{N, 1}-\frac{b_{N, N}}{N}\right)=-N\alpha_1+\alpha_k^N=-N\alpha_1+s^N.$$
On the other hand, $\pi_N=n_1+Nn_N=n_1=s$ so that 
$$\alpha_1=\frac{s^N-s}{N}.$$
$s>1$ follows from the assumption that $\alpha_1>0$.\\

$(3)$: The following is the table for $n=N+l$ for every $1\leq l\leq N-1$:
\begin{center}
\begin{tabular}{ |c|c|c|c|c|c|c|c| }
\hline
  $j_{N+1}$ & $j_{N}$ & $j_1$ & $i$ & mult. coeff.\\
\hline
 $1$ & & $l-1$ & $l$ & $l$\\ 
\hline
    & $1$  & $l$ & $l+1$ & $l+1$\\ 
\hline
   &  & $N+l$ & $N+l$ & $1$\\ 
\hline
\end{tabular}
\end{center}  

Therefore 
\begin{eqnarray*}
    \pi_{N+l}&=&(N+l)\sum^{N+l}_{i=1}\frac{(-1)^i}{i}
b_{N+l, i}=(N+l)\left((-1)^lb_{N+l, l}+(-1)^{l+1}b_{N+l, l+1}+(-1)^{N+l}\frac{b_{N+l, N+l}}{N+l}\right)\\
    &=&(N+l)\left((-1)^l(-\alpha_k)^{l-1}(-\alpha_{k-1})+(-1)^{l+1}(-\alpha_k)^l\alpha_1+(-1)^{N+l}\frac{(-\alpha_k)^{N+l}}{N+l}\right)\\
    &=&(N+l)\alpha_k^{l-1}\alpha_{k-1}-(N+l)\alpha_k^l\alpha_1+\alpha_k^{N+l}\\
    &=&\frac{-(N+l)s^l(s^N-s)+Ns^{N+l}}{N}+(N+l)s^{l-1}\alpha_{k-1}\\
    &=&\frac{-ls^{N+l}+(N+l)s^{l+1}}{N}+(N+l)s^{l-1}\alpha_{k-1}\\
\end{eqnarray*}

$(4)$: Apply $l=N-1$ to $(3)$. Since $\pi_{2N-1}\geq0$, we obtain
$$(2N-1)s^{N-2}\alpha_{k-1}\geq\frac{(N-1)s^{2N-1}-(2N-1)s^N}{N}.$$
Therefore
$$\alpha_{k-1}\geq \frac{(N-1)s^{N+1}}{N(2N-1)}-\frac{s^2}{N}$$
therefore
\begin{eqnarray*}
m_{N+l}&=&\frac{-ls^{N+l}+(N+l)s^{l+1}}{N}+(N+l)s^{l-1}\alpha_{k-1}\\
&\geq& \frac{-ls^{N+l}+(N+l)s^{l+1}}{N}+(N+l)s^{l-1}\left(\frac{(N-1)s^{N+1}}{N(2N-1)}-\frac{s^2}{N}\right)\\
&=&\frac{(N+l)(N-1)-l(2N-1)}{N(2N-1)}s^{N+l}=\frac{N^2-(l+1)N}{N(2N-1)}s^{N+l}\\
&=&\frac{N-(l+1)}{2N-1}s^{N+l}
\end{eqnarray*}

$(5)$: Recall from $(3)$ that we have the following table for $n=N+l$ for every $1\leq l\leq N-1$:
\begin{center}
\begin{tabular}{ |c|c|c|c|c|c|c|c| }
\hline
  $j_{N+1}$ & $j_{N}$ & $j_1$ & $i$ & mult. coeff.\\
\hline
 $1$ & & $l-1$ & $l$ & $l$\\ 
\hline
    & $1$  & $l$ & $l+1$ & $l+1$\\ 
\hline
   &  & $N+l$ & $N+l$ & $1$\\ 
\hline
\end{tabular}
\end{center}  

It is clear that the coefficient of $\alpha_j$'s in each term of $\pi_{2N+l}$ viewed as a polynomial in $\alpha_j$'s is an integer. If $l>2$, then $j_1\geq2$ for each term of $\pi_n$. Thus $\pi_n$ is divisible by $\alpha_k^2=s^2$ if $2<l\leq N-1$.

$(6)$: The following is the table for $n=2N+l$ for some $2\leq l\leq N-1$:
\begin{center}
\begin{tabular}{ |c|c|c|c|c|c|c|c| }
\hline
 $j_{2N}$ & $j_{N+1}$ & $j_{N}$ & $j_1$ & $k$ & mult. coeff.\\
\hline
$1$ & & & $l$ & $l+1$ & $l+1$\\ 
\hline
 & $2$ & & $l-2$ & $l$ & $\binom{l}{2}$\\ 
\hline
  & $1$ & $1$ & $l-1$ & $l+1$ & $l(l+1)$\\ 
\hline
  & $1$ & & $N+l-1$ & $N+l$ & $N+l$\\ 
\hline
  & & $2$ &  $l$ & $l+2$ & $\binom{l+2}{2}$\\ 
\hline
  & & $1$ & $N+l$ & $N+l+1$ & $N+l+1$\\ 
\hline
  & &  & $2N+l$ & $2N+l$ & $1$\\ 
\hline
\end{tabular}
\end{center}  

It is clear that the coefficients of $\alpha_j$'s in the terms of $\pi_{2N+l}$ corresponding to $k=l+1, N+l, N+l+1 2N+l$ are all integers. If $k=l$, then the coefficient is $(2N+l)(l-1)/2$, and this is an integer since one of the two factors of the numerator is always even. If $k=l+2$, the coefficient is $(2N+l)(l+1)/2$ and this is also an integer. Therefore, $\pi_{n}$ can be viewed as a polynomial over $\alpha_j$'s with integer coefficients. If $l>3$, then $j_1\geq 2$ for each term of $\pi_n$. This means that every term of $\pi_n$ is divisible by $\alpha_k^2=s^2$ if $3<l\leq N-1$.  \end{proof}

\begin{rem}
    It is also possible to compute $\pi_i$ by using Taylor series. First, take the logarithm of both sides of 
$$p(t)=\prod^m_{i=1}(1-t^i)^{n_i}.$$
to obtain
$$\log p(t)=\sum^m_{i=1}n_i\log(1-t^i).$$
By taking the derivatives of both sides, we have
$$\frac{p'(t)}{p(t)}=\sum^m_{i=1} n_i\cdot\frac{-it^{i-1}}{1-t^i}$$
Therefore, we obtain
$$\sum^m_{i=1}\frac{in_it^i}{1-t^i}=-\frac{tp'(t)}{p(t)}.$$
Since $n_i=0$ for all $i>m$, we may view the left-hand side as a Lambert series:
$$\sum^\infty_{i=1}\frac{in_it^i}{1-t^i}=-\frac{tp'(t)}{p(t)}.$$

We have
$$\sum^\infty_{i=1}\pi_it^i=-\frac{tp'(t)}{p(t)}$$
and hence we may compute $\pi_i$ by comparing the coefficients of $p(t)(\sum^\infty_{i=1}\pi_it^i)$ and those of $-tp'(t)$.
\end{rem}

\subsection{An upper bound of $k$ in terms of $n_i$'s for the case $m\geq N+1$}

The goal of this subsection is to bound $k$ by using $N$ and $n_i$'s where $1\leq i\leq N-1$ when $m\geq N+1$. By $(1)$ of Proposition \ref{seqp}, we know that $\pi_i=s^i$ for $1\leq i\leq N-1$ so that $n_i$ can be expressed using the powers of $s^i$. It further enables us to bound $k$ using $N$ and $s$.

We start with the polynomial
$$p(t)=\prod^m_{i=1}(1-t^i)^{n_i}$$
where $m\geq1$ and $n_1, \cdots, n_m$ are positive integers. Note that $n_{lN}=0$ for all $l\geq 1$ such that $1\leq lN\leq m$ by \cite[Proposition 3.3]{Kabbaj}. Therefore, it can be written as a product
$$p(t)=(1-t)^{n_1}(1-t^2)^{n_2}\cdots(1-t^{N-1})^{n_{N-1}}(1-t^{N+1})^{n_{N+1}}\cdots.$$
Consider the polynomial
$$\frac{p(t)}{\prod_{i=1}^{N-1}(1-t^i)^{n_i}}=\prod_{i=N+1}^m(1-t^i)^{n_i}\ \ \ \ \ \ \ (*).$$
The degree of this polynomial is given by 
$$kN+1-\sum^{N-1}_{i=1}in_i.$$
Since $m\geq N+1$, the number of the factors $1-t^i$ that appear on the right-hand side of $(*)$ can be written as 
$$\sum^{m}_{i=N+1}n_i=\sum^{m}_{i=1}n_i-\sum^{N-1}_{i=1}n_i$$
which is greater than 
$$\frac{21}{11}k+1-\sum^{N-1}_{i=1}n_i$$
by assumption $(2)$ of Proposition \ref{main}. Since $1-t^i$ has degree at least $N+1$ when $i\geq N+1$, we have 
$$(N+1)\left(\frac{21}{11}k+1-\sum^{N-1}_{i=1}n_i\right)< kN+1-\sum^{N-1}_{i=1}in_i.$$
By moving all the terms without $k$ on the left-hand side to the right-hand side, we obtain
$$\left(\frac{10}{11}N+\frac{21}{11}\right)k< (N+1)\sum^{N-1}_{i=1}n_i-\sum^{N-1}_{i=1}in_i-N=\sum^{N-1}_{i=1}(N+1-i)n_i-N.$$
Thus we have proved the following:
\begin{lem}\label{lockout}
    Suppose that $m\geq N+1$. Then we have 
$$\left(\frac{10}{11}N+\frac{21}{11}\right)k<\sum^{N-1}_{i=1}(N+1-i)n_i-N.$$
   In particular,
$$k<\frac{11}{10}\sum^{N-1}_{i=1}\frac{N+1-i}{N}n_i-\frac{11}{10}\leq\frac{11}{10}\sum^{N-1}_{i=1}n_i.$$
\end{lem}

\subsection{$N\geq7$ is impossible}

The goal of this subsection is to show that $N\geq 7$ is impossible and hence $N=5$ or $N=3$.

\subsubsection{The impossibility of the case $m\geq 2N-1$}\label{7m>2N-1imp}


\begin{cor}\label{7m<2N-1}
Let $N\geq 7$, $q(t)$ and $p(t)$ be polynomials as in Theorem \ref{main}. Then $m<2N-1$.
\end{cor}
\begin{proof}
It suffices to show that the case $m\geq 2N-1$ is impossible. By Definition \ref{piseq} and the Mobius inversion formula, we may bound $n_i$ from above as follows:
$$n_i=\frac{1}{i}\sum_{d|i}\mu\left(\frac{i}{d}\right)\pi_{d}\leq\frac{1}{i}\sum_{d|i} s^d\leq\frac{1}{i}\sum_{d=1}^is^d=\frac{s}{i}\cdot\frac{s^i-1}{s-1}\leq \frac{2}{i}s^i\leq 2s^i.$$
where we use the inequalities $s\geq2$ and $\frac{s}{2}\leq s-1$. 
By the second equality of Lemma \ref{lockout} we have
\begin{eqnarray*}
    kN+1&<&1+\frac{11}{10}N\sum^{N-1}_{i=1}n_i\leq1+\frac{11}{5}N\sum^{N-1}_{i=1}s^i\\
    &=&1+\frac{11}{5}N\cdot\frac{s}{s-1}(s^{N-1}-1))\leq1+\frac{22}{5}N(s^{N-1}-1))\leq\frac{22}{5}Ns^{N-1}
\end{eqnarray*}
On the other hand, since $m\geq 2N-1$, we have
$$kN+1=\sum_{i=1}^min_i\geq\sum_{i|2N-1}in_i=\pi_{2N-1}\geq\frac{s^{2N-2}}{2N-1}$$
by $(4)$ of Proposition \ref{seqp} for $l=N-2$.
Thus we obtain
$$\frac{s^{2N-2}}{2N-1}\leq \frac{22}{5}Ns^{N-1}.$$
By multiplying by $(2N-1)/s^{N-1}$, we have
$$s^{N-1}\leq \frac{22}{5}N(2N-1).$$
In particular
$$2^{N-1}\leq \frac{22}{5}N(2N-1)\ \ \ \ \ (**).$$
Note that
$$\frac{\frac{22}{5}(N+1)(2N+1)}{\frac{22}{5}N(2N-1)}=\frac{2N^2+3N+1}{2N^2-N}=1+\frac{4N+1}{2N^2-N}$$
and this is less than $2$ when $N>4$. Since $2^{11-1}=1024>1016.4=22\cdot 11\cdot (22-1)/5$, the inequality $(**)$ holds only when $N<10$. By assumption, $N$ is a prime greater than or equal to $7$ so that $N=7$. $\sqrt[6]{22\cdot 7\cdot 13/5}=2.71...<3$ thus $s=2$. Recall that
$$n_1=s, n_2=\frac{s^2-s}{2}, n_3=\frac{s^3-s}{3}, n_4=\frac{s^4-s^2}{4}, n_5=\frac{s^5-s}{5}, n_6=\frac{s^6-s^3-s^2+s}{6}$$
if $N=7$. The values of $n_i$'s and those of $in_i$'s for $s=2$ are the following: 

\begin{center}
\begin{tabular}{ |c|c|c|c|c|c|c|c| }
\hline
 $i$ & $1$ & $2$ & $3$ & $4$ & $5$ & $6$ \\
\hline
$n_i$ & $2$ & $1$ & $2$ & $3$ & $6$ & $9$ \\
\hline
$in_i$ & $2$ & $2$ & $6$ & $12$ & $30$ & $54$ \\ 
\hline
\end{tabular}
\end{center}

The table and Lemma \ref{lockout} show that
$$\frac{91}{11}k\leq 7n_1+6n_2+5n_3+4n_4+3n_5+2n_6-7=14+6+10+12+18+18-7=71$$
and hence $k\leq 7$.  However, $54=6n_6<\sum_{i=1}^m in_i=7k+1\leq 50$, which is a contradiction. 

Thus we conclude that $m<2N-1$ if $N\geq 7$.
\end{proof}

\subsubsection{The impossibility of the case $m<2N-1$}\label{Na2}

Recall the following improvement of Bertrand's postulate by Nagura:

\begin{thm}\cite{Nagura}
    Let 
    $$a_1:=2, a_2:=8, a_3:=9, a_4:=24, a_5:=25.$$
    For every integer $1\leq n\leq 5$ and every real number $x\geq a_n$, there exists a prime number $l$ such that
    $$x<l<\left(\frac{n+1}{n}\right)x.$$
\end{thm}
\begin{cor}\label{2N-1}
    Let $q(t)$ and $p(t)$ be polynomials as in Theorem \ref{main}. Then  $N=5$ or $N=3$.
\end{cor}
\begin{proof}
First suppose that $N\geq 11$. By Corollary \ref{7m<2N-1}, we know that $m<2N-1$. We apply the above theorem for $n=4$ and $x=2N+3$: the theorem asserts that for every real number $N \geq 11$, there exists a prime number $l$ such that
$$2N+3<l<\frac{5}{4}(2N+3).$$
Note that 
$$3N-1-\frac{5}{4}(2N+3)=\frac{2N-19}{4}\geq\frac{22-19}{3}=1>0.$$
Thus there exists a prime $l$ such that $$2N+3<l\leq3N-1.$$
Then Proposition \ref{seqp}(5) implies that $\pi_l$ is divisible by $s^2$. On the other hand, $l>2N+3>2N-1>m$ so
$$\pi_l=n_1+ln_l=n_1=s.$$
Therefore $s^2$ divides $s$ so that $s=1$, which contradicts $(2)$ of Proposition \ref{seqp}. Therefore $N=7$. Note that $m<2N-1$ still holds by Corollary \ref{7m<2N-1}.

Then we may apply the same argument for the prime $19$, which is between $2N+3=17$ and $3N-1=20$, to say that
$\pi_{19}=s$. Again, this $\pi_{19}$ must be divisible by $s^2$ so that $s=1$, which is a contradiction. Therefore, the assumption that $N\geq 7$ is false and we conclude that $N=5$ or $N=3$ since $N$ must be an odd prime less than $7$. 
\end{proof}

\begin{rem}
We may also take $n=5$ in the above argument. In this case,
$$3N-1-\frac{6}{5}(2N+3)=\frac{3N-23}{5}\geq\frac{33-23}{3}=\frac{10}{3}>0.$$
Therefore there exists a prime $l$ such that $$2N+3<l\leq3N-1.$$
\end{rem}


\subsection{If $N=5$ or $N=3$, then $m\leq N-1$}\label{53m>N-1}

By using the result from the previous subsection, we have eliminated the possibility of the case $N\geq7$. Since $N$ is an odd prime, it suffices to consider $N=3$ and $N=5$.

\begin{cor}\label{not35}
    Let $q(t)$ and $p(t)$ be polynomials as in Proposition \ref{main}. If $N=5$ or $N=3$, then $m\leq N-1$.
\end{cor}
\begin{proof}
Suppose toward contradiction that $m\geq N+1$. If $N=5$, then the first equality of Lemma \ref{lockout} implies that
\begin{eqnarray*}
    5k&<&\frac{71}{11}k< 5n_1+4n_2+3n_3+2n_4-5=5s+2(s^2-s)+s^3-s+\frac{s^4-s^2}{2}-5\\
    &=&\frac{1}{2}s^4+s^3+\frac{3}{2}s^2+2s-5
\end{eqnarray*}
so
$$5k+1<\frac{1}{2}s^4+s^3+\frac{3}{2}s^2+2s-4.$$
On the other hand,
$$5k+1=\sum^m_{i=1}in_i> n_1+2n_2+3n_3+4n_4=s+s^2-s+s^3-s+s^4-s^2=s^4+s^3-s$$
so that
$$s^4+s^3-s<5k+1<\frac{1}{2}s^4+s^3+\frac{3}{2}s^2+2s-4.$$
Therefore 
$$s^4-3s^2-6s+8=(s-1)(s-2)(s^2+3s+4)=(s-1)(s-2)\left(\left(s+\frac{3}{2}\right)^2+\frac{7}{4}\right)<0.$$
However, there is no integer $s$ satisfying the inequality. Therefore $N=3$.

Again, the first equality of Lemma \ref{lockout} yields
$$5k<\frac{51}{11}k< 3n_1+2n_2-3=3s+s^2-s-3=s^2+2s-3$$
so
$$k<\frac{s^2+2s-3}{5}.$$
On the other hand,
$$3k+1=\sum^m_{i=1}in_i> n_1+2n_2=s+s^2-s=s^2.$$
since $n_m>0$ with $m\geq4$. Therefore 
$$\frac{s^2-1}{3}<k.$$
Hence $5(s^2-1)<3(s^2+2s-3)$ so that
$$2s^2-6s+4=2(s^2-3s+2)=2(s-1)(s-2)<0$$
This is a contradiction since there is no integer $s$ satisfying this inequality. Thus $m<N+1$. By definition, $m$ is the largest positive integer such that $n_m>0$. We know that $n_N=0$ so that $m\leq N-1$. 
\end{proof}


\subsection{Proof that $N=3$ and $p(t)=(1-t)^2(1-t^2)$}

The following finally proves Theorem \ref{main}.

\begin{thm}\label{N+1neq5}
Let $q(t)$ and $p(t)$ be polynomials as in Theorem \ref{main}. Then $N=3$, $k=1$ and $p(t)=(1-t)^2(1-t^2)$.
\end{thm}
\begin{proof}
First, suppose $N=5$. Corollary \ref{not35} implies that $m\leq 4$. In particular, we have $n_6=n_7=0$. Proposition \ref{seqp}(3) for $l=1$ yields
$$\pi_6=\frac{-s^6+6s^2}{5}+6\alpha_{k-1}.$$
On the other hand, by $n_6=0$ and Proposition \ref{seqp}(1), we have
$$\pi_6=n_1+2n_2+3n_3=\pi_2+\pi_3-n_1=s^2+s^3-s$$
so 
$$6\alpha_{k-1}=\frac{s^6-6s^2}{5}+\pi_6=\frac{s^6-6s^2}{5}+s^2+s^3-s=\frac{s^6-6s^2+5s^2+5s^3-5s}{5}=\frac{s^6+5s^3-s^2-5s}{5}.$$

Similarly, Proposition \ref{seqp}(3) for $l=2$ yields
$$\pi_7=\frac{-2s^7+7s^3}{5}+7s\alpha_{k-1}.$$
On the other hand, by $n_7=0$ and Proposition \ref{seqp}(1), we have $\pi_7=n_1=s$
so
$$7\alpha_{k-1}=\frac{1}{s}\left(\frac{2s^7-7s^3}{5}+s\right)=\frac{2s^6-7s^2+5}{5}$$
Therefore
$$7s^6+35s^3-7s^2-35s=7(s^6+5s^3-s^2-5s)=6(2s^6-7s^2+5)=12s^6-42s^2+30$$
so
$$5s^6-35s^3-35s^2+35s+30=0.$$
The left-hand side is factorized as
$$35(s^6-7s^3-7s^2+7s+6)=35(s-1)(s+1)(s-2)(s^3+2s^2+5s+3).$$
Since $s$ is a positive integer, $s=1$ or $s=2$. The former leads to a contradiction to $(2)$ of Proposition \ref{seqp}. Suppose that $s=2$. Then 
$$n_1=\pi_1=2,\ n_2=\frac{\pi_2-n_1}{2}=\frac{s^2-s}{2}=\frac{4-2}{2}=1,\ n_3=\frac{\pi_3-n_1}{3}=\frac{s^3-s}{3}=\frac{8-2}{3}=2$$
and
$$n_4=\frac{\pi_4-\pi_2}{4}=\frac{s^4-s^2}{4}=\frac{16-4}{4}=3.$$
Therefore 
$$\deg p(t)=\sum^4_{i=1}in_i=1\cdot 2+2\cdot 1+3\cdot 2+4\cdot 3=22,$$
which contradicts that $\deg p(x)=5k+1$ for some positive integer $k$.

Therefore $N=3$. Again, Corollary \ref{not35} shows that $m\geq 2$ and in particular $n_4=n_5=0$. Proposition \ref{seqp}(3) for $l=1$ implies that
$$\pi_4=\frac{-s^4+4s^2}{3}+4\alpha_{k-1}.$$
On the other hand, $n_4=0$ so that $\pi_4=n_1+2n_2=\pi_2=s^2$. Therefore we obtain
$$4\alpha_{k-1}=\frac{s^4-4s^2}{3}+s^2=\frac{s^4-4s^2+3s^2}{3}=\frac{s^4-s^2}{3}.$$

Furthermore, Proposition \ref{seqp}(4) for $l=2$ implies that
$$\pi_5=\frac{-2s^5+5s^3}{3}+5s\alpha_{k-1}.$$
Again, $n_5=0$ so that $\pi_5=n_1=\pi_1=s>0$. Therefore
we obtain
$$5\alpha_{k-1}=\frac{1}{s}\left(\frac{2s^5-5s^3}{3}+s\right)=\frac{2s^4-5s^2+3}{3}.$$
Hence
$$5s^4-5s^2=5(s^4-s^2)=4(2s^4-5s^2+3)=8s^4-20s^2+12.$$
so
$$3s^4-15s^2+12=0.$$
The left-hand side can be factorized as
$$3s^4-15s^2+12=3(s^2-4)(s^2-1).$$
Thus $s$ must be either $1$ or $2$. 
Again, the former leads to a contradiction $(2)$ of Proposition \ref{seqp}.

Therefore $s=2$. We have
$$n_1=\pi_1=2\ \text{and}\ n_2=\frac{\pi_2-n_1}{2}=\frac{s^2-s}{2}=\frac{4-2}{2}=1.$$ 
Thus $p(t)=(1-t)^2(1-t^2)$. Also, $3k+1=\deg p(t)=4$ so $k=1$ and $q(t)=1+\alpha_1t=1+st=1+2t$. 
\end{proof}


\textbf{Acknowledgments.} The author thanks Abdourrahmane Kabbaj for helpful conversations. In addition, the author would like to thank Manuel L. Reyes for his feedback and advice. 

\bibliographystyle{plain}
\bibliography{soresearch}
\end{document}